\pdfoutput=1
\documentclass[a4paper]{amsart}
\usepackage{diff-style,diff-macros}

\title[Fundamental groups of framed embedded circles in 4-manifolds]
{On fundamental groups of spaces of framed embeddings of a circle in a 4-manifold}

\author{Danica Kosanovi\'c}
\address{ETH Z\"urich, Department of Mathematics, R\"amistrasse 101, 8092 Z\"urich, Switzerland}
\curraddr{Universität Bern, Mathematisches Institut (MAI), Alpeneggstrasse 22, 3012 Bern, Switzerland}
\email{danica.kosanovic@unibe.ch}

\date{\today}

\begin{document}

\begin{abstract}
    Motivated by recent results on diffeomorphisms of 4-manifolds, this paper investigates fundamental groups of spaces of embeddings of $S^1\tm D^3$ in 4-manifolds. The majority of work goes into the case of framed immersed circles.
\end{abstract}

\maketitle

%%%~~~~~~~~
~~~~~~~~~~%%%
\section{Introduction}\label{sec:intro}

\subsection{Motivation}
\label{subsec-intro:motivation}
This paper is motivated by recent results of Botvinnik, Budney, Gabai, Gay, Hartman, Watanabe~\cite{Watanabe-dim-4,Botvinnik-Watanabe,Gay,Gay-Hartman,Budney-Gabai}, that give various constructions of diffeomorphisms of 4-manifolds.  
% starting from classes in $\pi_1(\Emb(\nu\S^1,\M\#\,\S^3\tm\S^1);\nu c)$.
In \cite{K-diffeos-from-graspers} these constructions are expressed in terms of a certain map $\ps$ from the fundamental group
\[
    \pi_1(\Emb(\nu\S^1,\X);\nu c)
\]
of framed embeddings of a circle into $\X$, for $\nu\S^1\coloneqq\S^1\tm\D^3$.

On one hand, Watanabe's~\cite{Watanabe-dim-4} and Gay's diffeomorphisms~\cite{Gay} arise from the case $\X=\M\#\,\S^3\tm\S^1$ and $c=\{pt\}\tm\S^1$, which has an interpretation in terms of pseudo-isotopies of $\M$ (diffeomorphisms of $\M\tm[0,1]$ that are the identity on $\M\tm\{0\}$). Cerf described pseudo-isotopies in terms of one-parameter families of generalised Morse functions on $\M\tm[0,1]$. A 1-2 family has a single birth and death of a pair of critical points of indices 1 and 2; in between the family is described by a loop of framed embeddings of a single circle in $\M\#\,\S^3\tm\S^1$. In fact, a diffeomorphism of $\M$ admits a 1-2 pseudo-isotopy if and only if it is in the image of the parameterised surgery map
\[
    \ps\colon \pi_1(\Emb(\nu\S^1,\M\#\,\S^3\tm\S^1);\nu c)\to\pi_0\Diffp(\M). 
\]
On the other hand, Budney--Gabai's barbell implantation diffeomorphisms~\cite{Budney-Gabai} are in the image of the more general map
\[
    \ps_S\colon \pi_1(\Emb(\nu\S^1,\M_{\nu S});\nu c)\to\pi_0\Diffp(\M),
\]
where $\X=\M_{\nu S}$ is the surgery on an embedded sphere $S\colon\S^2\hra \M$, and $\nu c$ is the $\S^1\tm\D^3$ added in the surgery (note that $\X=\M\#\,\S^3\tm\S^1$ is precisely the case when $S$ is unknotted). 
In fact, such a map $\ps_S$ exists for an arbitrary 4-manifold $\X$ and a framed circle $\nu c\colon\nu\S^1\hra \X$. Namely, let $\M\coloneqq X_{\nu c}$ be the surgery on $\nu c$, and let $\nu S\subset\M$ be the added $\S^2\tm\D^2$, so that $\M_{\nu S}$ is diffeomorphic to $X$. Therefore, computing the fundamental group of $\Emb(\nu\S^1,\X)$ is of interest for any $\X$.

The literature on this and related computational problems is not extensive. The fundamental groups $\pi_1(\Emb(\S^1,\X);c)$ of spaces of smooth (non-framed) embeddings of a circle into a $4$-manifold $\X$, for a basepoint $c\colon\S^1\hra\X$, have been studied by Arone and Szymik~\cite{AS}, Moriya~\cite{Moriya}, by Gabai in~\cite{Gabai-disks}, and Budney and Gabai in~\cite{Budney-Gabai}. We have studied them in~\cite{KT-highd,K-Dax}, following the seminal work of Dax~\cite{Dax} in more general settings.

In this paper 
% we extend the results from~\cite{K-Dax} in the case $d=4$ in two directions: we study the framings, and the case $\X=\M\#\,\S^3\tm\S^1$ in detail.
% First, 
we study the fundamental group of the space $\Emb(\nu\S^1,\X)$ of embeddings of $\nu\S^1\coloneqq\S^1\tm\D^3$ in $\X$, the framed analogue of $\Emb(\S^1,\X)$. We compare them via the map forgetting the framing:
    $\forg\colon \Emb(\nu\S^1,\X)\ra \Emb(\S^1,\X)$,
that is, precomposing with $\S^1\tm\{0\}\hra\S^1\tm\D^3$.
% Our first main result says that $\pi_1\forg$ has kernel $\Z/2$ if $\X$ is spin or almost spin ($w_2^s=0$). Otherwise, when $\X$ is totally nonspin ($w_2^s\neq0$), the map $\pi_1\forg$ has kernel either $\Z/2$ or trivial. Here $w_2^s\colon\pi_2\X\to\Z/2$ is the spherical Stiefel--Whitney class (see Definition~\ref{def:ws}). 
% For some applications of these results we refer the reader to \cite{K-diffeos-from-graspers}.
The main work in fact goes into the study of the analogous map $\forg_{\Imm}\colon \Imm(\nu\S^1,\X)\to \Imm(\S^1,\X)$ for immersions,
as we shall explain in the rest of the introduction.

\subsection{The key players}
\label{subsec-intro:summary}

Let us recall that a knot $K\colon\S^1\hra \X$ in a 3-manifold $\X$ has $\Z$-many framings, unless $K$ has a geometric dual (an embedded 2-sphere in $\X$, intersecting $K$ in exactly one point), in which case it has only two nonisotopic framings; see for example~\cite{Chernov}.

On the other hand, for an embedding $c\colon\S^1\hra \X$ in a 4-manifold $\X$ there are two possible framings (nonzero sections of its normal bundle) because $\pi_1SO_3\cong\Z/2$. Whether these framings are isotopic (homotopic through nonzero sections) turns out to depend on the following invariants of $X$.
\begin{defn}
\label{def:ws}
    The \emph{spherical Stiefel--Whitney class}
\[
w_2^s\colon\pi_2\X\ra\Z/2
\]
    denotes the restriction to the image of the Hurewicz map $\pi_2\X\to H_2(\X;\Z)$ of the homomorphism
\[
w_2^h\colon H_2(\X;\Z)\ra\Z/2,
\]
    that we call the \emph{homological Stiefel--Whitney class}. This is the image of the second Stiefel-Whitney class $w_2=w_2(\X)$ under the canonical map $H^2(\X;\Z/2)\to \Hom(H_2(\X;\Z),\Z/2)$. If $w_2^s(b)=1$ then $b\in\pi_2\X$ is represented by an immersed sphere of odd self-intersection. If $w_2^h(\sigma)=1$ then $\sigma\in H_2(\X;\Z)$ is represented by an embedded surface of odd self-intersection.
\end{defn}
\begin{rem}
    The last two sentences come from the usual game: the Stiefel--Whitney class of the pullback of the tangent bundle of $X$ to a surface $\Sigma\to X$ is the Euler number of its normal bundle modulo two: $w_2(TX|_\Sigma)=w_2(\nu(\Sigma)) + w_2(T\Sigma)=e(\nu\Sigma)\pmod2$.
    
    We will see in Lemma~\ref{lem:w2} that $w_2^s$ can be identified with $w_2(\wt{\X})$, so it makes sense to use the following terminology, often found in the literature on classification of 4-manifolds: $X$ is \emph{spin} ($w_2X=0$) or \emph{almost spin} ($w_2X\neq0$, $w_2^sX=0$) or \emph{totally nonspin} ($w_2^sX\neq0$).
\end{rem}
Back to our question, let us fix a framed circle $\nu c$ and let $\tw_{\nu c}\in\pi_0\Emb(\nu\S^1,X)$ be the framing of $c$ that differs from $\nu c$ by a \emph{full twist}, that is, for $A\colon\S^1\to SO_3$ generating $\pi_1SO_3=\Z/2$ we define
\begin{equation}\label{eq:tw}
    \tw_{\nu c}\colon\nu\S^1=\S^1\tm\D^3\hra \X, \quad 
    \tw_{\nu c}(x,v)=\nu c(x,A_x(v)).
\end{equation}
Note that $\tw_{\nu c}$ is trivial (that is, equal to $\nu c$) if and only if there is a loop of nonframed embeddings based at $c$, whose ambient extension starts with $\nu c$ but returns as $\tw_{\nu c}$. See also Remark~\ref{rem:tw}.

In Theorem~\ref{thm-intro:framed-emb} we will see that $\nu c$ and $\tw_{\nu c}$ are isotopic if and only if $w_2^s\neq0$, or $w_2^s=0$ but there is an immersed torus in $\X$ with nontrivial $w_2^h$ and that contains $\bc$.
Such a torus can be foliated into a loop of embedded circles based at $c$, but when trying to lift to framed embeddings, we get only a path that starts with one framing and ends with the other. Therefore, in these cases the cokernel of $\pi_1(\forg,c)$ is equal to $\Z/2$, and otherwise it is trivial.

Independently of whether the two framings are isotopic, we can pick one, say $\nu c$, and form a loop $\rot_{\nu c}$ of framed circles based at $\nu c$ by \emph{letting all the normal disks complete a full rotation}. More precisely, define $\rot_{\nu c}\in\pi_1(\Emb(\nu\S^1,X);\nu c)$ as the loop of framed embeddings
\begin{equation}\label{eq:rot}
    \rot_{\nu c}(t)\colon\nu\S^1=\S^1\tm\D^3\hra \X, \quad 
    \rot_{\nu c}(t)(x,v)=\nu c(x,A_t(v))
\end{equation}
for $t\in\S^1$ and the generating loop $A\colon\S^1\to SO_3$ as above. See also Remark~\ref{rem:rot}.

As we shall see in Theorem~\ref{thm-intro:framed-emb}, this loop is isotopic to the trivial one if and only if there is a class $b\in\pi_2\X$ with $w_2^s(b)=1$, that is fixed by the action of $\bc\in\pi_1\X$, and for which a certain invariant
\begin{equation}\label{eq:dax-whisk}
    \begin{tikzcd}
        \dax^{whisk}_c\colon \Fix_{\bc}(\pi_2\X)
            \rar{}
        & \faktor{\Z[\pi_1\X\sm\{1\}]}{\dax_u(\pi_3(\X\sm\D^4))}
    \end{tikzcd}
\end{equation}
vanishes. In Remark~\ref{rem:dax} we recall the definition of this homomorphism from~\cite{K-Dax}, noting that details will not be of much relevance in this paper. 

\subsection*{Conventions}
We fix basepoints $e\in\S^1$ and $c\in\Emb(\S^1,X)$ and $c(e)\in\X$. We denote by $\bc=[c]\in\pi_1Z=\pi_1(Z;c(e))$ the homotopy class of $c$. 
We write $\nu\S^1\coloneqq\S^1\tm\D^3$ and fix a framing $\nu c\in\Emb(\nu\S^1,X)$ of $c$. We consider the evaluation map
\[
    \ev_e\colon\Emb(\S^1,X)\to X,\quad\gamma\mapsto K(e).
\]
Denote by $\bc\cdot b$ the canonical action of $\bc\in\pi_1Z$ on $b\in\pi_nZ$. We consider
\[
    \Fix_{\bc}(\pi_nZ)\coloneqq\{h\in\pi_nZ:h=\bc\cdot h\}.
\]
For $n=1$ the canonical action is the conjugation action $\bc\cdot h=\bc h \bc^{-1}$, so the subgroup
\[
    \Fix_{\bc}(\pi_1Z)\coloneqq\{h\in\pi_1Z:h=\bc h \bc^{-1}\}<\pi_1Z
\]
is precisely the centraliser of $\bc\in\pi_1Z$.  Note that in this case $y\vee c\colon\S^1\vee\S^1\hra X$ extends to a map $y\tm c\colon \S^1\tm\S^1\hra X$, that defines a class $[y\tm\bc]\in H_2X$.

We foliate $\S^2$ by the meridians, that is, arcs with fixed endpoints at the north and south poles. We fix one of these as our basepoint arc.

\begin{rem}\label{rem:dax}
    First, let $u\colon\D^1\hra X\sm\D^4$ be obtained from $c$ by removing a small 4-ball around the basepoint $c(e)$, giving $u=c|_{\S^1\sm nbhd(e)}$. Given a loop $H_1$ of embedded arcs based at $u$, and a homotopy $H$ from $H_1$ to the constant loop $H_0$ through loops $H_t$ of immersed arcs we have the Dax invariant $\Da(H)$. This counts signed double points of immersed arcs that occur during $H$; see \cite{Dax,Gabai-disks,K-Dax,KT-highd}. 
    We define the Dax homomorphism
    \[
        \dax_u\colon \pi_3(\X\sm\D^4)\ra \Z[\pi_1\X\sm\{1\}]
    \]
    by thinking of $a\colon\S^3\to\X\sm\D^4$ as a 2-parameter self-homotopy $H_a$ of $u$, so $\dax_u(a)\coloneqq\Da(H_a)$.
    
    Finally, for a class $b\in\pi_2\X$ with $c\cdot b=b$, we define 
    \[
        \dax^{whisk}_c(b)=\Da(H_b)
    \]
    for the following homotopy $H_b$ of immersed arcs. 
    First, represent $b\in\pi_2\X$ by $b\colon[0,1]\to \Omega X=\Map((\S^1,e),(X,c(e)))$ with $b(0)=b(1)=\const_{c(e)}$. Second, use ambient isotopy extension to lift this to a family $B\colon [0,1]\to \Omega\Emb(\S^1,X)$ such that $B(1)=\const_c$ and $\ev_e\circ B=b$. Then $\delta^{whisk}_c(b)\coloneqq B(0)|_{\S^1\sm nbhd(e)}\in\pi_1(\Embp(\D^1,X\sm\D^4);u)$ is trivial when included in $\pi_1(\Emb(\S^1,X);c)$. What is more, the equality $\bc\cdot b=b$ means that there is also a homotopy from $\delta^{whisk}_c(b)$ to $u$ through immersed neat arcs; this is our desired $H_b$. We refer to \cite[Thm~4.5]{K-Dax} for details.
% In general, the cokernel of $\pi_1\forg$ is also either trivial or $\Z/2$, depending -- curiously -- on the homological Stiefel--Whitney class  $w_2^h\colon H_2\X\to\Z/2$ (see Definition~\ref{def:ws}). 
% In Theorem~\ref{thm-intro:framed-emb} we give a complete list of conditions. 
\end{rem}

\subsection{Framed immersions}
\label{subsec-intro:framed-imm}
In fact, in order to explain what happens more conceptually, in Theorem~\ref{thm-intro:framed-imm} we first study the analogous map for immersions:
\[
    \forg_{\Imm}\colon \Imm^{\fr}\coloneqq\Imm(\nu\S^1,\X)\ra \Imm(\S^1,\X)\eqqcolon\Imm.
\]
% Intuitively, $w_2^s$ plays a role because we are trying to foliate a 2-sphere $b$ in $\X$ into a 1-parameter family of framed embeddings by foliating $b$. A generic representative $B$ of $b$ has double points; their count modulo two agrees with the Euler number, and is measured by $w_2^s(b)$. If this is trivial, then we can make the desired extension.

Smale~\cite{Smale} showed that taking the unit derivative gives a homotopy equivalence $\deriv$ from the space of immersions $\Imm(\S^1,\X)$ to the space of all loops $\Lambda \S \X\coloneqq\Map(\S^1,\S \X)$ in the unit sphere subbundle of the tangent bundle. 
Similarly, the space $\Imm(\nu\S^1,\X)$ is homotopy equivalent to the space of all loops $\Lambda \Fr\X=\Map(\S^1,\Fr\X)$ in the frame bundle $\Fr\X$, the space of tuples of four mutually orthogonal unit vectors in the tangent bundle of $\X$. The first vector in the frame corresponds to the tangent vector to the circle, whereas the remaining three vectors describe the normal data.
Therefore, for all $n\geq0$ we have
\begin{align}
    \pi_n(\Imm(\S^1,\X); c) &\cong\pi_n(\Lambda \S \X;\deriv c),\label{eq:Smale}\\
    \pi_n(\Imm(\nu\S^1,\X); \nu c) &\cong\pi_n(\Lambda \Fr\X;\deriv \nu c).\label{eq:Smale-framed}
\end{align}
For this reason, in Section~\ref{subsec:Lambda} we first study homotopy groups of free loop spaces in low degrees. In Section~\ref{subsec:nonframed-imm} we present the folklore description of $\pi_1(\Imm(\S^1,\X);c)$ as a central group extension
\begin{equation}\label{eq:imm}
\begin{tikzcd}[column sep=large]
        \faktor{\pi_2\X}{b=\bc\cdot b}\rar[tail]{\foliate_{\bc}} 
        & \pi_1(\Imm(\S^1,\X);c)\rar[two heads]{\pi_1\ev_e} 
        & \Fix_{\bc}(\pi_1\X).
\end{tikzcd}
\end{equation}
We give a sufficient condition for this to split in Theorem~\ref{thm:imm}.
The foliation map $\foliate_{\bc}$ in~\eqref{eq:imm} sends $b\in\pi_2\X$ to a loop of immersed circles based at $c$, that foliate a representative $B\colon\S^2\to \X$ of $b$. Namely, we can assume that $B$ is an immersed sphere, and our loop starts by pushing a piece of $c$ close to $B$, until it agrees with the basepoint meridian arc in $B$ (see Conventions). It then twirls around, passing through all meridians, and then comes back to $c$.

In Section~\ref{subsec:Fr} we study frame bundles $\Fr\X$. For an oriented 4-manifold $X$ we prove in Corollary~\ref{cor:w2-dim4} that one of the following alternatives holds.
    \begin{enumerate}[leftmargin=75pt]
        \item[(totally nonspin)] 
            $w_2^s\neq0$ $\iff$ There is an immersed sphere in $\X$ of odd self-intersection.
        \item[(h-nonspin)] $w_2^s=0$, $w_2^h\neq0$ 
            $\iff$ All immersed spheres have even self-intersection numbers, but there is an immersed torus in $\X$ of odd self-intersection.
        \item[(h-spin)] $w_2^h=0$, $w_2\neq0$
            $\iff$ There is no section of $H_1(\Fr\X)\sra H_1(\X)$.
            $\iff$ There exists $\alpha\in H_1(\X)$ such that $2\alpha=0$ but $2\wt{\alpha}\neq0\in H_1(\Fr\X)$ for any framed lift $\wt{\alpha}$.
        \item[(spin)] $w_2=0$ 
            $\iff$ $H_1(\Fr\X)\cong\Z/2\oplus H_1\X$.
            $\iff$ There is a consistent choice of framing for each class in $H_1\X$.
    \end{enumerate}
Interestingly, this gives a different proof of \cite[Thm.1]{Kirby-Melvin-Teichner}; see also Remark~\ref{rem:KMT}.
    
Finally, in Section~\ref{subsec:proof} we compute $\pi_1(\Imm(\nu\S^1,\X); \nu c)$, as follows. Recall that $\bc=[c]\in\pi_1\X$.
\begin{mainthm}[Corollary~\ref{cor:framed-imm}]
\label{thm-intro:framed-imm}
    For an oriented  smooth 4-manifold $\X$ and a framed embedded circle $\nu c\colon\nu\S^1=\S^1\tm\D^3\hra\X$, we have the following exact sequences.
\begin{enumerate}
 \item Assume $w_2^s\neq0$ ($\X$ totally nonspin).
    \begin{enumerate}
\item 
    If there exists $b\in\pi_2\X$ with $w_2^s(b)=1$ and $b=\bc\cdot b$, (that is, there is an immersed sphere in $\X$ of odd self-intersection, and fixed by the action of $\bc$), then
    \[
    \begin{tikzcd}[column sep=small]
        \quad%\Z/2
            %\rar{0}
        & \pi_1(\Imm^{\fr};\nu c)
            \rar[tail] 
        & \pi_1(\Imm;c)
            \rar[two heads] 
        & \Z/2, 
            %\rar{0}
        & \nu c=\tw_{\nu c}.
    \end{tikzcd}
    \]
\item 
    Otherwise,
    \[
    \begin{tikzcd}[column sep=small]
        \Z/2
            \rar[tail]{\rot_{\nu c}}
        & \pi_1(\Imm^{\fr};\nu c)
            \rar
        & \pi_1(\Imm;c)
            \rar[two heads]
        & \Z/2, 
            % \rar{0}
        & \nu c=\tw_{\nu c}.
    \end{tikzcd}
    \]
    \end{enumerate}
\item 
Assume $w_2^s=0$ ($\X$ spin or almost spin).
    \begin{enumerate}
\item
If there exists $y\in\Fix_{\bc}(\pi_1\X)$ such that $w_2^h([y\tm\bc])\neq0$ (that is, there is an immersed torus in $\X$ of odd self-intersection, and that contains $\bc$), then
    \[\begin{tikzcd}[column sep=small]
        \Z/2
            \rar[tail]{\rot_{\nu c}}
        & \pi_1(\Imm^{\fr};\nu c)
            \rar
        & \pi_1(\Imm;c)
            \rar[two heads] 
        & \Z/2, 
            %\rar
        & \nu c=\tw_{\nu c}.
    \end{tikzcd}
    \]
\item 
    Otherwise,
    \[\begin{tikzcd}[column sep=small]
        \Z/2
            \rar[tail]{\rot_{\nu c}}
        & \pi_1(\Imm^{\fr};\nu c)
            \rar[two heads] 
        & \pi_1(\Imm;c),
            %\rar{0} 
        &\qquad\hfill
        & \nu c\neq\tw_{\nu c}.
    \end{tikzcd}
    \]
    \end{enumerate}
\end{enumerate} 
    Moreover, $\rot_{\nu c}$ splits if $w_2=0$ ($\X$ spin). More generally, this splits whenever the extension $\Z/2\to\Fix_{\nu c}(\pi_1\Fr\X)\sra\Fix_c(\pi_1\X)$ does.
\end{mainthm}

\subsection{Framed embeddings}
\label{subsec-intro:framed-emb}
Having completed the study of framed immersions, in Section~\ref{sec:framed-emb} we turn to framed \emph{embeddings} and the analogous forgetful map
\[
    \forg\colon \Emb^{\fr}\coloneqq\Emb(\nu\S^1,\X)\ra \Emb(\S^1,\X)\eqqcolon\Emb.
\]
It turns out that we need to work just a little bit more: the only difference to Theorem~\ref{thm-intro:framed-imm} will be the appearance of the Dax homomorphism $\dax^{whisk}_c$ from \eqref{eq:dax-whisk}.
\begin{mainthm}
\label{thm-intro:framed-emb}
    For an oriented  smooth 4-manifold $\X$ and a framed embedded circle $\nu c\colon\nu\S^1=\S^1\tm\D^3\hra\X$, we have the following exact sequences.
\begin{enumerate}
 \item Assume $w_2^s\neq0$ ($\X$ totally nonspin).
    \begin{enumerate}
\item 
    If there exists $b\in\pi_2\X$ with $w_2^s(b)=1$ and $b=\bc\cdot b$ and $\dax^{whisk}_c(b)=0$, then
    \[
    \begin{tikzcd}[column sep=small]
        \quad%\Z/2
            %\rar{0}
        & \pi_1(\Emb^{\fr};\nu c)
            \rar[tail] 
        & \pi_1(\Emb;c)
            \rar[two heads] 
        & \Z/2, 
            %\rar{0}
        & \nu c=\tw_{\nu c}.
    \end{tikzcd}
    \]
\item 
    Otherwise,
    \[
    \begin{tikzcd}[column sep=small]
        \Z/2
            \rar[tail]{\rot_{\nu c}}
        & \pi_1(\Emb^{\fr};\nu c)
            \rar
        & \pi_1(\Emb;c)
            \rar[two heads]
        & \Z/2, 
            % \rar{0}
        & \nu c=\tw_{\nu c}.
    \end{tikzcd}
    \]
    \end{enumerate}
\item 
Assume $w_2^s=0$ ($\X$ spin or almost spin).
    \begin{enumerate}
\item
If there exists $y\in\Fix_{\bc}(\pi_1\X)$ such that $w_2^h([y\tm\bc])\neq0$, then
    \[\begin{tikzcd}[column sep=small]
        \Z/2
            \rar[tail]{\rot_{\nu c}}
        & \pi_1(\Emb^{\fr};\nu c)
            \rar
        & \pi_1(\Emb;c)
            \rar[two heads] 
        & \Z/2, 
            %\rar
        & \nu c=\tw_{\nu c}.
    \end{tikzcd}
    \]
\item 
    Otherwise,
    \[\begin{tikzcd}[column sep=small]
        \Z/2
            \rar[tail]{\rot_{\nu c}}
        & \pi_1(\Imm^{\fr};\nu c)
            \rar[two heads] 
        & \pi_1(\Imm;c),
            %\rar{0} 
        &\qquad\hfill
        & \nu c\neq\tw_{\nu c}.
    \end{tikzcd}
    \]
    \end{enumerate}
\end{enumerate} 
    Moreover, $\rot_{\nu c}$ splits if $w_2=0$ ($\X$ is spin). More generally, it splits whenever the extension $\Z/2\to\Fix_{\nu c}(\pi_1\Fr\X)\sra\Fix_c(\pi_1\X)$ does.
\end{mainthm}

When $\bc=1$ this result simplifies as follows. In this case $b=\bc\cdot b$ and $\dax^{whisk}_c(b)=0$ for all $b\in\pi_2\X$, so either $w_2^s\neq0$ and $\rot_{\nu c}=0$, or $w_2^s=0$ and $\rot_{\nu c}\neq0$. Moreover,  $\rot_{\nu c}$ splits precisely when $w_2=0$, thanks to Lemma~\ref{lem:FrX}, and $\bc=1$ implies $w_2^h([y\tm\bc])=0$ for all $y\in\pi_1Y$.
\begin{cor}
\label{cor:trivial-framings}
    Fix an oriented  smooth 4-manifold $\X$ and $c\colon\S^1\hra \X$ which is \emph{nullhomotopic}, $\bc=1\in\pi_1\X$.
    If $\X$ is totally nonspin ($w_2^s\neq0$), we have
    \[\begin{tikzcd}[column sep=small]
        \quad
        & \pi_1(\Emb^{\fr}; \nu c)
            \rar[tail] 
        & \pi_1(\Emb;c)
            \rar[two heads] 
        & \Z/2, 
            %\rar{0}
        & \nu c=\tw_{\nu c}.
    \end{tikzcd}
    \]
    Otherwise, if $\X$ is spin or almost spin ($w_2^s=0$), then
    \[\begin{tikzcd}[column sep=small]
        \Z/2
        \rar[tail]{\rot_{\nu c}}
        & \pi_1(\Emb^{\fr}; \nu c)
            \rar[two heads] 
        & \pi_1(\Emb;c),
            %\rar{0} 
        & \quad
        & \nu c\neq\tw_{\nu c}.
    \end{tikzcd}
    \]
    Moreover, $\rot_{\nu c}$ splits precisely when $X$ is spin ($w_2=0$).
\end{cor}

As another example, consider $\X=\M\#\,\S^3\tm\S^1$ for a spin manifold $M$, $w_2(M)=0$. Then $w_2(\X)=w_2(\M)+w_2(\S^3\tm\S^1)=0$ as well, and we are in both ``otherwise'' cases of Theorem~\ref{thm-intro:framed-emb}.

\begin{cor}
\label{cor-intro:framed-stab-case}
    Let $\M$ be a \emph{spin}  smooth 4-manifold, and $c\coloneqq\{pt\}\tm\S^1\subset \M\#\,\S^3\tm\S^1$. Then $\nu c$ and $\tw_{\nu c}$ are not isotopic, and there is an isomorphism
    \[
        \pi_1\big(\Emb(\nu\S^1,\M\#\,\S^3\tm\S^1);\nu c\big)\cong\Z/2\tm\pi_1(\Emb(\S^1,\M\#\,\S^3\tm\S^1);c).
    \]
\end{cor}

In fact, one can show that the situations is similar for any $M$, but we leave this for future work. % (there is no class with $b=\bc\cdot b$ and for any $y\in\Fix_\bc(\pi_1\X)=\langle\bc\rangle$ we have $y\wedge\bc=0$). 

\begin{rem}
    For readers interested in generalising these results to manifolds of dimension $d\geq4$, we expect a similar story. The first interesting homotopy group is $\pi_{d-3}(\Emb(\nu\S^1,X);\nu c)$ and the two key $\Z/2$'s appearing in our exact sequences need to be replaced by $\pi_nSO(d)$ for $n=d-2$, $d-3$, $d-4$ (compare to \eqref{eq:les} and the discussion after it); these belong to the stable range and depend on $n$ modulo $8$, by Bott periodicity.
\end{rem}

\subsection*{Acknowledgements} I wish to thank Mark Powell for useful comments. I am grateful to the referee for valuable suggestions and corrections.

%%%~~~~~~~~~~~~~~~~~~%%%
\section{Immersions}
\label{sec:imm}

\subsection{Free loop spaces}
\label{subsec:Lambda}

For a connected space $Z$ consider the space $\Lambda Z=\Map(\S^1,Z)$ of all loops in $Z$, and the fibration that evaluates loops at the basepoint $e\in\S^1$:
\[
    \ev_e\colon\Lambda Z\sra Z,\quad x\mapsto x(e).
\]
Each fibre is homotopy equivalent to the based loop space $\Omega Z\coloneqq\Map((\S^1,e),(Z,*))$. 
Moreover, we have an isomorphism $\foliate\colon\pi_{n+1}Z\xrightarrow{\sim}\pi_n(\Omega Z;\const)$ that foliates an $(n+1)$-sphere by an $n$-parameter family of (based) circles.
The following is a standard result, revisited in \cite[Sec.\ 4.1]{K-Dax}.
\begin{prop}\label{prop:Lambda}
    Let $Z$ be any connected topological space and $c\colon\S^1\to Z$, so that $c(e)=*\in Z$ is the basepoint. Then for $\ev_e$ and $n\geq1$ the connecting map
    \[
        \delta_\Lambda\colon \pi_{n+1}Z\ra 
        \pi_n(\Omega Z;\bc)\cong\pi_n(\Omega Z;\const)\cong\pi_{n+1}Z,
    \]
    is given by $b\mapsto b-\bc\cdot b$.
    In the case $n=0$ the connecting map $\conjbc\colon\pi_1Z\to\pi_0(\Omega Z)\xrightarrow{\sim}\pi_1Z$ acts via $h\mapsto h\bc h^{-1}$. 
\end{prop}

\begin{rem}
    The second paragraph was stated incorrectly in \cite[Sec.\ 4.1]{K-Dax}, using the action $h\mapsto h\bc h^{-1}\bc^{-1}$ (although the statement was correct in \cite[Thm.\ C.II]{K-Dax}). As explained there, in the case $n\geq1$ the map is the Whitehead product $[-,\bc]$, since the basepoint $c\in\Omega Z$ is first changed to $\const$, in order to make identification with $\pi_{n+1}Z$. However, for $n=0$ and $\pi_0(\Omega\Z)\cong\pi_1Z$ we do not choose a basepoint, so there is no additional multiplication by $\bc^{-1}$.
\end{rem}

\begin{cor}\label{cor:Lambda}
    For a connected topological space $Z$ there is a bijection $\pi_0\Lambda Z\cong\pi_1Z/\sim$ with the set of conjugacy classes of $\pi_1Z$ (so $x\sim z\iff (\exists y\in\pi_1Z)\, z=yxy^{-1}$). Moreover, for a component of $\Lambda Z$ represented by $\bc\in\pi_1Z$ and any $n\geq1$ there is a central group extension:
\begin{equation}\label{eq:Lambda-ext}
\begin{tikzcd}[column sep=large]
    \faktor{\pi_{n+1}Z}{b=\bc\cdot b}\rar[tail]{\foliate_{\bc}} 
    & \pi_n(\Lambda Z;\bc)\rar[two heads]{\pi_n\ev_e} & \Fix_{\bc}(\pi_nZ).
    \qedhere
\end{tikzcd}
\end{equation}
\end{cor}

The extensions~\eqref{eq:Lambda-ext} might be nontrivial even in the case $n=1$ and $\bc=1\in\pi_1Z$. In that case there is a right splitting, but as we shall now see the resulting semi-direct product is given by the natural action of $\pi_1Z$ on $\pi_2Z$, which is nontrivial in general.

\begin{lem}\label{lem:Lambda-action}
    The action of $\Fix_{\bc}(\pi_1Z)$ on $\faktor{\pi_2Z}{b=\bc\cdot b}$ in the extension \eqref{eq:Lambda-ext} for $n=1$ is given by the natural action of $\pi_1Z$ on $\pi_2Z$. Thus, it is classified by an element of $H^2(\Fix_{\bc}(\pi_1Z);\faktor{\pi_2Z}{b=\bc\cdot b})$.
\end{lem}
\begin{proof}
    We need to compute the action of $\mathsf{g}=[g]\in\Fix_{\bc}(\pi_1Z)$ on $[B]\in\pi_1(\Lambda Z;\bc)$ coming from $b\in\pi_2Z$. Note that an element in $\pi_1(\Lambda Z;c)$ is represented by a map $B\colon[0,1]\tm[0,1]\to Z$, $B(s,\theta)= B_s(\theta)$, with $B_0=B_1=c$ and $B_s(0)=B_s(1)=c(e)$ for all $s\in[0,1]$. 
    
    First, if $\ol{B}\colon[0,1]^2\to Z$ represents $b\in\pi_2Z$, with $\partial\ol{B}=\const_{c(e)}$, then $B$ can be obtained by putting $\ol{B}$ into a bigger square, and connecting each point in the outer boundary to the corresponding point of $\partial B$ via $c$ (on the sides) or $\const_{c(e)}$ (on the top and bottom).
    Second, we pick a class $[G]\in\pi_1(\Lambda Z;\bc)$ that projects down to $\pi_1\ev_e(G)=g$, so $G\colon[0,1]\tm[0,1]\to Z$, $G(s,\theta)=G_s(\theta)$ with $G_0=G_1=c$ and $G_s(e)=g(s)$ for all $s\in[0,1]$. Such a $G$ exists precisely because $cg$ is homotopic to $gc$. Now, we need to show that the class $[GBG^{-1}]\in\ker(\pi_1\ev_e)$ represents $\mathsf{g}\cdot b\in\pi_2Z$.

    On one hand, $GBG^{-1}$ is given by concatenating loops in the $s$-direction: for a fixed $\theta$ we have a loop $G_{\bull}(\theta)B_{\bull}(\theta)G_{\bull}(\theta)^{-1}$ in $Z$. In other words, we are stacking the squares horizontally. On the boundary we see $cc^{-1}$, so we also add a nullhomotopy of this, to obtain a square with $c(e)$ on the boundary, describing a class in $\pi_2Z$.
    
    On the other hand, the class $\mathsf{g}\cdot b$ is given by putting the square $\ol{B}$ into a bigger square and connecting each point in the outer boundary to the corresponding point of $\partial B$ via $g$. A path in $\pi_2Z$ to this picture from the one for $GBG^{-1}$ can be obtained by cancelling the two copies of $G$.
\end{proof}

\begin{lem}\label{lem:rot}
    If $\Fix_{\bc}(\pi_1Z)$ is equal to the subgroup $\langle \bc\rangle<\pi_1Z$ generated by $\bc$, then the extension~\eqref{eq:Lambda-ext} for $n=1$ has a splitting $\rot$, that is, $\pi_1\ev_e\circ\rot=\Id$. This implies that we have an isomorphism
\[
    \pi_1\ev_e\tm\unrot\colon\;
        \pi_1(\Lambda Z;\bc)
        \xrightarrow{\cong} 
        \langle\bc\rangle
        \tm\faktor{\pi_2Z}{b=\bc\cdot b},
\]
    where $\unrot(\alpha)$ is the sphere traced out by $\alpha\cdot(\rot\circ\pi_1\ev_e(\alpha))^{-1}$. In particular, the fundamental group of the component $\bc\in \Lambda Z$ is abelian.
\end{lem}
\begin{proof}
    A splitting $\rot\colon \Fix_{\bc}(\pi_1\X)=\langle\bc\rangle\to\pi_1(\Lambda Z;\bc)$ can be defined by sending $\bc$ to the loop given by rotating $c\colon[0,1]\to Z$ once: $(s,\theta)\mapsto c(\theta+s\pmod{1})$. Then $c^k$ is sent to $k$ full rotations, and $\rot$ is a group homomorphism. Applying $\ev_e$ to this clearly gives back $\bc^k$, so $\rot$ is indeed a splitting. Thus, we have a semi-direct product $\pi_1(\Lambda Z;\bc)
        \cong \langle\bc\rangle
            \rtimes\faktor{\pi_2Z}{b=\bc\cdot b}$.
    Now note that the action of $\bc^k$ is trivial on ${\pi_2Z}/{b=\bc\cdot b}$, as this action has already been modded out in this target. Thus, this is in fact the product $\langle\bc\rangle\tm\faktor{\pi_2Z}{b=\bc\cdot b}$.
\end{proof}

%~~~~~~~~~~
\subsection{Nonframed immersions}
\label{subsec:nonframed-imm}

Let $\X$ now be a connected smooth 4-manifold.
Recall from \eqref{eq:Smale} that $\pi_i(\Imm(\S^1,\X);c)\cong\pi_i(\Lambda\S \X;c)$ via the derivative map $\deriv$, omitted from the notation from now on. We can in fact replace $\S \X$ by $\X$, as follows.
\begin{lem}\label{lem:imm-Lambda}
    There is a bijection $\pi_0\Imm(\S^1,\X)\cong\pi_0(\Lambda \X)$, an isomorphism $\pi_1(\Imm(\S^1,\X);c)\cong\pi_1(\Lambda \X;c)$, and a surjection $\pi_2(\Imm(\S^1,\X);c)\sra\pi_2(\Lambda \X;c)$.
\end{lem}
\begin{proof}
    Use either general position or~\eqref{eq:Smale}, the extension~\eqref{eq:Lambda-ext}, and the fact that the bundle $\S^3\hra \S \X\sra \X$
    has a 2-connected fibre, so $\pi_i\S \X\cong\pi_i \X$ for $i=0,1,2$ and $\pi_3\S \X\sra \pi_3 \X$.
\end{proof}

In particular, Corollary~\ref{cor:Lambda} and Lemma~\ref{lem:rot} imply the following.
\begin{thm}\label{thm:imm}
    For a connected smooth 4-manifold $X$ we have $\pi_0\Imm(\S^1,\X)\cong 
    \faktor{\pi_1\X}{\sim}$ and
\[\begin{tikzcd}[row sep=tiny]
        \faktor{\pi_2\X}{b=\bc\cdot b}\rar[tail] 
        & \pi_1(\Imm(\S^1,\X);\bc)\rar[two heads] 
        & \Fix_{\bc}(\pi_1\X).
    \end{tikzcd}
\]
Moreover, 
if $\bc$ is self-centralising, that is, $\Fix_{\bc}(\pi_1Z)=\langle \bc\rangle$, then the last extension is trivial:
\begin{equation}\label{eq:imm-pi-1}
   \pi_1\ev_e\tm\unrot\colon\pi_1(\Imm(\S^1,\X);c)\cong\pi_1(\Lambda \X;\bc)\cong \langle \bc\rangle\tm\faktor{\pi_2\X}{b=\bc\cdot b}.
\end{equation}
\end{thm}

\begin{rem}
   A class of manifolds with $\bc$ self-centralising is $X=M\#\S^1\tm\S^3$ with $c=\S^1\tm\{pt\}$.
\end{rem}

%~~~~~~~~~~
\section{Framed immersions}
\label{sec:framed-imm}

\subsection{Second Stiefel-Whitney classes}
\label{subsec:sw}
In this section $X$ is any smooth manifold.
Recall from Definition~\ref{def:ws} that for the second Stiefel-Whitney class $w_2(\X)\in H^2(\X;\Z/2)$, we denote by
    \[
        w_2^h\colon H_2(X)=H_2(\X;\Z)\to\Z/2
    \]
the image of $w_2(X)$ under the canonical map $H^2(\X;\Z/2)\to \Hom(H_2(\X),\Z/2)$, and by
    \[
        w_2^s\colon\pi_2\X\to\Z/2
    \]
the restriction of $w_2^h$ to $h(\pi_2\X)\subseteq H_2(\X)$.

\begin{lem}
    For a manifold $X$ and its universal cover $\wt{p}\colon\wt{X}\to X$ there is a commutative diagram
\begin{equation}\label{eq:univ-cover-diag}
\begin{tikzcd}
        \Ext(H_1(\X),\Z/2)\rar[tail]\dar
        & H^2(\pi_1\X;\Z/2)\rar[two heads]{com}\dar{b^*} 
        & \Hom(H_2(\pi_1\X),\Z/2)\dar\\
        \Ext(H_1(\X),\Z/2)\rar[tail]\dar 
        & H^2(\X;\Z/2)\rar[two heads]\dar{\wt{p}^*} 
        & \Hom(H_2(\X),\Z/2)\dar{\Hom(\wt{p}_*,\Z/2)}\\
        0 \rar[tail]
        & H^2(\wt{\X};\Z/2)\rar[tail,two heads] 
        & \Hom(H_2(\wt{\X}),\Z/2)
        \end{tikzcd}
    \end{equation}
    in which all rows and columns are exact sequences.
\end{lem}
\begin{proof}
    The Universal Coefficients Theorem applied to $B\pi_1\X=K(\pi_1\X,1)$ and $X$ and $\wt{X}$ gives the three rows of this diagram, from the top to the bottom respectively. The maps between them commute by the naturality of this theorem.

    It is a theorem of Heinz Hopf~\cite{Hopf} that $\pi_2\X\to H_2\X\sra H_2(\pi_1\X)$ is an exact sequence, and since the functor $\Hom(\cdot,\Z/2)$ is left exact, we have the exactness of the rightmost column, using that $H_2(\wt{\X})\cong\pi_2\X$. 
    In fact, Hopf's sequence can be extracted from the homology Serre spectral sequence for the fibration sequence 
        $\begin{tikzcd}[cramped,column sep=small]
        \wt{\X}\rar{\wt{p}} & \X \rar{b} & B\pi_1\X
        \end{tikzcd}$.
    The analogous argument in cohomology gives the exactness of the middle column.
\end{proof}
     
We now describe the map $com$ at the top right of \eqref{eq:univ-cover-diag}.
\begin{lem}\label{lem:Hopf}
    For a group $G$ we have 
    \[
        H_2(G)\cong\ker ([-,-]\colon G\wedge G\ra [G,G]),
    \]
    where $G\wedge G\coloneqq G\otm G/\langle g\otm g:g\in G\rangle$ is the nonabelian exterior square of a group, $[G,G]$ is the commutator subgroup of $G$, and $[-,-](g\wedge h)=[g,h]$.
    
    If $w\in H^2(G;\Z/2)$ corresponds to an extension $\Z/2\ira E\sra G$ then 
    \[
        com(w)\colon \ker(G\wedge G\to [G,G])\ra\Z/2
    \]
    sends $g_1\wedge g_2$ to $[\wt{g}_1,\wt{g}_2]\in\ker(E\sra G)\cong\Z/2$, using arbitrary lifts of $g_1,g_2\in G$ to $\wt{g}_1,\wt{g}_2\in E$.
\end{lem}
\begin{proof}
    The description $H_2(G)=\ker (G\wedge G\to [G,G])$, the so-called \emph{Schur multiplier}, is essentially due to Hopf. Miller~\cite{Miller} gave the description of $com$.
\end{proof}

We now justify the name ``almost spin'' for manifolds with $w_2(X)\neq0,w_2^s(X)=0$, usually used for non-spin manifolds with spin universal cover, $w_2(X)\neq0,w_2(\wt{\X})=0$.

\begin{lem}
\label{lem:w2}
    For a manifold $\X$ and its universal cover $\wt{p}\colon\wt{\X}\to \X$ we have $w_2^s(\X)=w_2(\wt{\X})$.
\end{lem}
\begin{proof}
    Let us show that in the diagram~\eqref{eq:univ-cover-diag} we have
    \[
    \begin{tikzcd}[column sep=small,row sep=small]
        w_2\rar[mapsto]\dar[mapsto][swap]{\wt{p}^*}
        & w_2^h\dar[mapsto]{\Hom(\wt{p}_*,\Z/2)}\\
        w_2(\wt{\X})\rar[mapsto]{\cong}
        & w_2^s.
    \end{tikzcd}
\]
    The middle horizontal map on the right sends $w_2$ to $w_2^h$ by definition, and $\Hom(\wt{p}_*,\Z/2)$ sends $w_2^h$ to a map $\pi_2\X\cong H_2(\wt{\X})\to\Z/2$, which is $w_2^s$ by definition. The map $\wt{p}^*$ sends $w_2$ to $w_2(\wt{\X})$ by the naturality of Stiefel--Whitney classes, so the bottom isomorphism indeed sends $w_2(\wt{\X})$ to $w_2^s$.
\end{proof}

\begin{cor}
\label{cor:w2}
    For a manifold $X$ one of the following alternatives holds.
    \begin{enumerate}[leftmargin=50pt]
        \item[(totally nonspin)] 
            $w_2^s\neq0$.
        \item[(almost spin)] 
            $w_2^s=0$ and $w_2\neq0$
        $\iff w_2=b^*(w_E)\neq0$ for a class $w_E\in H^2(\pi_1\X;\Z/2)$,
that corresponds to a nontrivial central extension $\Z/2\ira E\sra\pi_1\X$; then $w_2^h=com(w_E)$.
        \begin{enumerate}[leftmargin=20pt]
            \item[(h-nonspin)] $w_2^s=0$ and $w_2^h\neq0$ 
            $\iff$ there exist $g_1,g_2\in\pi_1\X$ with $[g_1,g_2]=1$ and
            $[\wt{g}_1,\wt{g}_2]\neq0$, for the lifts to $E$. In particular, $E^{ab}\cong H_1X$.
            \item[(h-spin)] $w_2^h=0$ and $w_2\neq0$
            $\iff$ $E$ is the pullback along $\pi_1\X\sra H_1X$ of the nontrivial extension $\Z/2\ira E^{ab}\sra H_1X$.
        \end{enumerate}
        \item[(spin)] $w_2=0$.
    \end{enumerate}
\end{cor}

Note that the names in the brackets should only used if additionally $X$ is orientable, $w_1(X)=0$.

\subsection{The frame bundle}
\label{subsec:Fr}
For an \textit{oriented} smooth \textit{4-manifold} let us consider the fibration sequence
\[
    SO_4\ra \Fr\X\ra \X.
\]
The bottom part of the long exact sequence is given by
\begin{equation}\label{eq:FrX-conn-map}
    \begin{tikzcd}
        \pi_2\X\rar &\Z/2\rar & \pi_1\Fr\X\rar[two heads] &\pi_1\X.
    \end{tikzcd}
\end{equation}
\begin{lem}
\label{lem:FrX}
    For an oriented 4-manifold $\X$ the connecting map in \eqref{eq:FrX-conn-map} is equal to the spherical Stiefel--Whitney class $w_2^s\colon\pi_2\X\to\Z/2$. Moreover, one of the following holds.
    \begin{enumerate}
        \item (totally nonspin) $w_2^s\neq 0$ $\iff$ $\pi_1\Fr\X\cong\pi_1\X$.
        \item (almost spin) $w_2^s=0$, $w_2\neq0$ $\iff$ there is a nonsplit extension
            $\Z/2\ira\pi_1\Fr\X\sra \pi_1\X$.
        \item (spin) $w_2=0$ $\iff$ $\pi_1\Fr\X\cong\Z/2\tm\pi_1\X$.
    \end{enumerate}
    In particular, if $w_2^s=0$ then $w_2(X)=w_E$ precisely classifies  $\Z/2\ira\pi_1\Fr\X\sra\pi_1\X$.
\end{lem}
\begin{proof}
It is a standard fact that the above fibration sequence deloops to
\[
    \Fr\X\ra \X\ra BSO_4,
\]
where the second map is the classifying map of the tangent bundle of $\X$. Since we have $\pi_2BSO_4\cong\pi_1SO_4\cong\Z/2$, we can postcompose this map with the canonical map $BSO_4\to K(\Z/2,2)$, to obtain the fibration $w_2\colon \X\to K(\Z/2,2)$, classified by the second Stiefel--Whitney class of $\X$, by one of its definitions. Therefore, we have a commutative diagram of long exact sequences:
    \[\begin{tikzcd}
        \pi_2\X\dar[equals]\rar 
        & \Z/2\rar\dar[equals] 
        & \pi_1\Fr\X \rar[two heads]\dar 
        & \pi_1\X \dar[equals]\\
        \pi_2\X \rar{w_2^s} 
        & \Z/2 \rar 
        & \pi_1 \hofib(w_2) \rar[two heads] 
        & \pi_1\X,
    \end{tikzcd}
    \]
    where $\hofib(w_2)$ is the homotopy fibre of $w_2$.
    We conclude that $\pi_1\Fr\X\cong\pi_1\hofib(w_2)$ and the connecting map in~\eqref{eq:FrX-conn-map} is indeed $w_2^s$.
    
    Now, if $w_2^s\neq0$ then $\Z/2$ is killed by the connecting map and we have $\pi_1\Fr\X\cong\pi_1\X$, giving the first statement. If $w_2^s=0$, then in the resulting extension the action of $\pi_1\X$ on $\Z/2$ is trivial, since $\Z/2$ has no nontrivial automorphisms. Note that a section $\pi_1\X\to\pi_1\Fr\X$ precisely corresponds to a trivialisation of $T\X$ over the 1-skeleton that extends over the 2-skeleton. Therefore, this extension is classified by a class $w\in H^2(B\pi_1\X;\Z/2)$ which corresponds to $w_2(\X)\in H^2(\X;\Z/2)$ under a map $\X\to B\pi_1\X$ inducing an isomorphism on $\pi_1$. This gives the other two statements.
\end{proof}

% \begin{lem}
% \label{lem:FrX-H}
% Let $\X$ be a 4-manifold as before.
%     \begin{enumerate}
%         \item If $w_2^h\neq 0$, then $H_1(\Fr\X)\cong H_1\X$.
%         \item If $w_2^h=0$, then there is an extension
%         \[
%             \Z/2\ira H_1(\Fr\X)\sra H_1\X.
%         \]
%         classified by $w_2$.
%     \end{enumerate}
% \end{lem}
% \begin{proof}
%     By the proof of Lemma~\ref{lem:FrX} we have $H_1(\Fr\X)\cong H_1(\hofib(w_2))$. We use the Serre spectral sequence for the fibration sequence $B(\Z/2)\ira\hofib(w_2)\sra \X$. On the 1-line we have $\Z/2$ and $H_1(\X)$, and the only possible differential is $H_2(\X)\to \Z/2$. This is precisely given by $w_2^h$: it is a transgression (see \cite[Sec.6.2]{McCleary}) which lifts a 2-chain in $\X$ to a relative chain in $(\hofib(w_2),\Z/2)$ (so tries to frame it) and then takes the boundary.

%     Thus, we have $H_1(\Fr\X)\cong H_1(\hofib(w_2)\cong H_1(\X)$ if $w_2^h\neq0$, and otherwise there is an extension $\Z/2\ira H_1\Fr\X\sra H_1\X$ for $w_2^h=0$. Moreover, if $w_2=0$ this splits because in that case we have $\pi_1\Fr\X\cong\Z/2\tm\pi_1\X$ by Lemma~\ref{lem:FrX}.
% \end{proof}

Using Corollary~\ref{cor:w2} and Lemma~\ref{lem:FrX} we can compute the first homology group $H_1(\Fr\X)$ as well.
\begin{cor}\label{cor:w2-FrX}
    For an oriented 4-manifold $X$ one of the following alternatives holds, where all mentioned extensions are nonsplit.
        \begin{enumerate}[leftmargin=80pt]
        \item[(totally nonspin)] 
            $w_2^s\neq 0$ $\iff$ $\pi_1\Fr\X\cong\pi_1\X$ and $H_1(\Fr\X)\cong H_1\X$.
        \item[(h-nonspin)] 
            $w_2^s=0$, $w_2^h\neq0$ $\iff$ $\Z/2\ira\pi_1\Fr\X\sra \pi_1\X$ and $H_1(\Fr\X)\cong H_1\X$.
        \item[(h-spin)] 
            $w_2^h=0$, $w_2\neq0$ $\iff$  $\Z/2\ira\pi_1\Fr\X\sra \pi_1\X$ and $\Z/2\ira H_1(\Fr\X)\sra H_1\X$.
        \item[(spin)] 
            $w_2=0$ $\iff$ $\pi_1\Fr\X\cong\Z/2\tm\pi_1\X$ and $H_1(\Fr\X)\cong\Z/2\oplus H_1\X$.
    \end{enumerate}
\end{cor}

Moreover, we can give the following geometric interpretation.
\begin{cor}
\label{cor:w2-dim4}
    For an oriented 4-manifold $X$ one of the following alternatives holds.
    \begin{enumerate}[leftmargin=80pt]
        \item[(totally nonspin)] 
            $w_2^s\neq0$ $\iff$ There is an immersed sphere in $\X$ of odd self-intersection.
        \item[(h-nonspin)] $w_2^s=0$, $w_2^h\neq0$ 
            $\iff$ All immersed spheres have even self-intersection numbers, but there is an immersed torus in $\X$ of odd self-intersection.
        \item[(h-spin)] $w_2^h=0$, $w_2\neq0$
            $\iff$ There is no section of $H_1(\Fr\X)\sra H_1(\X)$.
            $\iff$ There exists $\alpha\in H_1(\X)$ such that $2\alpha=0$ but $2\wt{\alpha}\neq0\in H_1(\Fr\X)$ for any framed lift $\wt{\alpha}$.
        \item[(spin)] $w_2=0$ 
            $\iff$ $H_1(\Fr\X)\cong\Z/2\oplus H_1\X$.
            $\iff$ There is a consistent choice of framing for each class in $H_1\X$.
        \end{enumerate}
\end{cor}
\begin{proof}
    Any class in $H_2(\X;\Z)$ is represented by a map $\Sigma\to\X$ of an oriented surface into a 4-manifold. This is generically an immersion, and the value of $w_2(\Sigma)$ is equal to $w_2$ of its normal bundle (as the tangent bundle of an oriented surface has even Euler number, so trivial $w_2$). Thus, $w_2^h([\Sigma])$ is the mod 2 reduction of the normal Euler number of $\Sigma\subset\X$. On the other hand, the normal Euler number can be computed by pushing $\Sigma$ off itself and counting signed intersections, which is by definition the self-intersection number $\Sigma\cdot\Sigma$.
    This implies the first statement. 
    
    For the second, recall that $w_2^s=0$, $w_2^h\neq0$ is by Corollary~\ref{cor:w2} equivalent to the existence of elements $g_1,g_2\in\pi_1\X$ that commute, $[g_1,g_2]=1$, but their lifts $\wt{g}_1,\wt{g}_2\in\pi_1(\Fr\X)$ do not, $[\wt{g}_1,\wt{g}_2]\neq0$. The first condition means that the map $g_1\vee g_2\colon\S^1\vee\S^1\hra X$ extends to a map of a torus $g_1\tm g_2\colon\S^1\tm\S^1\to X$. By general position, this can be homotoped into an immersed torus, in which our framed loops $\wt{g}_1,\wt{g}_2$ form a 1-skeleton. Then $w_2^h([g_1\tm g_2])=com(w_{\pi_1\Fr\X})(g_1\wedge g_2)=[\wt{g}_1,\wt{g}_2]\neq0$ precisely means that this torus cannot be framed. By the previous paragraph, this is equivalent to having odd self-intersection. 

    Now assume $w_2^h=0$, so we have an extension $\Z/2\ira H_1(\Fr\X)\sra H_1\X$, which is either trivial or not. For a set of generators $\alpha\in H_1\X$ we can pick lifts $\wt\alpha\in H_1(\Fr\X)$ --- that is, we frame $T\X|_\alpha$. This defines a section if and only if for every 2-torsion class we have $2\wt\alpha=0$ --- that is, along the surface that bounds $2\alpha$ we can extend the framing of $TX$.
\end{proof}

Note that if $w_2^h\neq0$ then $\X$ has odd intersection form, and otherwise it is even.

\begin{rem}\label{rem:KMT}
    This gives a different proof of \cite[Thm.1]{Kirby-Melvin-Teichner}, because $H_1(\Fr\X)$ is the set of all framed 1-submanifolds in $\X$ modulo framed bordism; this is equivalent to their normally framed version $\mathbb{F}_1(\X)$ because 1-manifolds are orientable.
    % It would be interesting to recover their description of $\mathbb{F}_2(\X)$ using our approach.
\end{rem}

%~~~~~~~~~~
\subsection{Framed immersions}
\label{subsec:framed-imm}

There are (vertical) maps of (horizontal) fibration sequences
\begin{equation}\label{eq:diag-imm-Lambda}
    \begin{tikzcd}[row sep=small]
        \forg_{\Imm}^{-1}(c)
            \rar[hook]\dar[sloped]{\sim} 
        & \Imm(\nu\S^1,\X)\rar{\forg_{\Imm}}
            \dar[sloped]{\sim} 
        & \Imm(\S^1,\X)\dar[sloped]{\sim},
            %\ar[shift left,bend left,looseness=1.5]{dd}
            \\
        \Lambda SO_3
            \rar[hook]\dar 
        & \Lambda \Fr\X
            \rar{\Lambda pr_1} \dar[equals]
        & \Lambda \S \X,
            \dar
            \\
        \Lambda SO_4
            \rar[hook] 
        & \Lambda \Fr\X 
            \rar{\Lambda pr} 
        & \Lambda \X.
    \end{tikzcd}
\end{equation}
Here $\forg_{\Imm}$ evaluates at $\S^1\tm\{0\}\subseteq\S^1\tm\D^3$, and is a fibration. At the top, the vertical maps in the middle and on the right are weak homotopy equivalences by \eqref{eq:Smale} and \eqref{eq:Smale-framed}. This implies the equivalence on the fibres: a framing of $c$ is a point in $\forg_{\Imm}^{-1}(c)$, and it can be identified with a free loop in $V_3(\R^3)\simeq O_3$, which is the fibre of $pr_1\colon \Fr\X\to V_1(\X)=\S \X$. Since $\X$ is oriented and we assume immersions are orientation preserving, we can choose our basepoint frame for $c(e)$ in $SO_3<O_3$, so that the fibre is $\Lambda SO_3$.
Similarly, the bundle $pr\colon \Fr\X\to \X$ has fibre the Stiefel manifold $V_4(\R^4)\simeq O_4$, but the components we are interested are $\Lambda SO_4$.

From~\eqref{eq:diag-imm-Lambda} we obtain the following diagram of maps between horizontal exact sequences:
\begin{equation}\label{eq:les}
    \begin{tikzcd}[column sep=8pt, row sep=small, scale cd=0.83]
        \pi_2(\Imm; c)
            \rar\dar[phantom,sloped]{\cong} &
        \pi_1(\forg_{\Imm}^{-1}(c);\nu)
            \rar\dar[phantom,sloped]{\cong} & 
        \pi_1(\Imm^{\fr};\nu c)
            \rar\dar[phantom,sloped]{\cong} &
        \pi_1(\Imm;c)\rar
            \dar[phantom,sloped]{\cong} &
        \pi_0(\forg_{\Imm}^{-1}(c))
            \rar\dar[phantom,sloped]{\cong} & 
        \pi_0(\Imm^{\fr}) 
            \rar[two heads]\dar[equals] & 
        \pi_0(\Imm)\dar[phantom,sloped]{\cong}\\
        \pi_2(\Lambda \S \X; c)
            \rar\dar[two heads] &
        \pi_1(\Lambda SO_3;\nu)
            \rar\dar[phantom,sloped]{\cong} & 
        \pi_1(\Lambda \Fr\X;\nu c)
            \rar\dar[equals] &
        \pi_1(\Lambda \S \X;c)\rar
            \dar[phantom,sloped]{\cong} &
        \pi_0\Lambda SO_3
            \rar\dar[phantom,sloped]{\cong} & 
        \pi_0\Lambda \Fr\X 
            \rar[two heads]\dar[equals] & 
        \pi_0\Lambda \S \X\dar[phantom,sloped]{\cong}\\
        \pi_2(\Lambda \X;\bc)
            \rar{\delta_2} & 
        \Z/2
            \rar &
        \pi_1(\Lambda \Fr\X;\nu c)
            \rar &
        \pi_1(\Lambda \X; c)
            \rar{\delta_1} & 
        \Z/2\rar &
        \pi_0\Lambda \Fr\X 
            \rar[two heads] & 
        \pi_0\Lambda \X
    \end{tikzcd}
\end{equation}
Here we have applied several facts. First, Lemma~\ref{lem:imm-Lambda} gives us isomorphisms $\pi_i(\Lambda\S \X)\cong\pi_i(\Lambda \X)$ for $i=0,1$ and a surjection for $i=2$. Second, we apply Corollary~\ref{cor:Lambda} to $Z=SO_3$, namely:
\[\begin{tikzcd}[row sep=small, column sep=small]
        \pi_1(\Omega SO_3;\nu)
            \rar\dar[phantom,sloped]{\cong} & 
        \pi_1(\Lambda SO_3;\nu)
            \rar &
        \pi_1(SO_4)
            \rar{\circlearrowleft^{\nu}}[swap]{0}\dar[phantom,sloped]{\cong} & 
        \pi_0(\Omega SO_3)
            \rar\dar[phantom,sloped]{\cong} &
        \pi_0(\Lambda SO_3) 
            \rar & 
        \pi_0(SO_3)\dar[phantom,sloped]{\cong}\\
        0 &
        &
        \Z/2 &
        \Z/2 &
        &
        0
    \end{tikzcd}
\]
Finally, this computation also holds for $SO_4$ instead of $SO_3$ since they have the same $\pi_0,\pi_1,\pi_2$. The generator of $\pi_1SO_3$ is given by $A\colon\S^1\to SO_3$ where $A_t$ is the rotation about the $z$-axis by $2\pi t$ of the basepoint frame, for $t\in[0,1]$ and $\S^1=[0,1]/\partial$.

\begin{rem}[The twist]\label{rem:tw}
    Looking at \eqref{eq:les} we see that $\Z/2\to\pi_0\Imm^{\fr}$ corresponds to the composite $\Z/2\cong\pi_0\Omega SO_3\to\pi_0\Lambda SO_3\to\pi_0\Lambda \Fr\X$. If $\nu c\in\pi_0\Imm^{\fr}$ is the basepoint, then the image of $1\in\Z/2$ is realised by
\[
        \tw_{\nu c}\colon\nu\S^1=\S^1\tm\D^3\hra \X, \quad 
        \tw_{\nu c}(x,v)=\nu c(x,A_x(v)).
\]
    as in~\eqref{eq:tw}. In other words, $\tw_{\nu c}\in\pi_0\Imm^{\fr}$ is a framing of $c$ that differs from $\nu c$ by a full twist; note that we rotate along the circle.
\end{rem}

\begin{rem}[The normal rotation]\label{rem:rot}
    Similarly, $\Z/2\to\pi_1(\Imm^{\fr};\nu c)$ corresponds to the composite of $\pi_1\ev_e\colon\pi_1(\Lambda SO_3;\nu)\to \pi_1SO_3\cong\Z/2$ and $\pi_1(\Lambda SO_3;\nu)\to\pi_1(\Lambda \Fr\X;\nu c)\cong\pi_1(\Imm^{\fr};\nu c)$. Since $\pi_1\ev_e$ is an isomorphism, we can pick any lift of the generating loop $A$ (the full rotation of the frame at the basepoint $\nu(e)$). We choose the simultaneous rotation of frames at all points of~$\nu$.
    
    Therefore, $1\in\Z/2$ is mapped to $\rot_{\nu c}\in\pi_1(\Imm^{\fr};\nu c)$ given at time $t\in[0,1]$ by
\[
        \rot_{\nu c}(t)\colon\nu\S^1=\S^1\tm\D^3\hra \X, \quad 
        \rot_{\nu c}(t)(x,v)=\nu c(x,A_t(v)).
\]
    % and $\psi\colon\S^1\to [0,1]$ is a bump function, with $\psi(e)=1$ and $\psi=0$ on some small interval around $e\in\S^1$. 
    as in~\eqref{eq:rot}.
    In other words, $\rot_{\nu c}\in\pi_1(\Imm^{\fr};\nu c)$ is the loop of framed immersions (in fact embeddings) that fully rotates all normal disks; note that we rotate along the time direction.
\end{rem}

\subsection{The connecting maps}
\label{subsec:proof}

To compute the homotopy groups of $\Imm(\nu\S^1,\X)\simeq\Lambda \Fr\X$ in low degrees it remains to describe the maps $\delta_2,\delta_1$ in \eqref{eq:les}. We use the conjugation action $\circlearrowleft_{\bc}(y) = y\bc y^{-1}$ from Proposition~\ref{prop:Lambda}, the foliation map $\foliate_c$ from \eqref{eq:Lambda-ext}, the maps $\tw_{\nu c}$ and $\rot_{\nu c}$ from Remarks~\ref{rem:tw},~\ref{rem:rot}.

\begin{lem}
\label{lem:lem1}
    For an oriented 4-manifold $\X$ and $\nu c\colon\nu\S^1\hra \X$ 
    there is a commutative diagram 
    \begin{equation}
    \label{eq:lem1}
    \begin{tikzcd}[column sep=20pt, scale cd=0.9]
        \Fix_{\bc}(\pi_2\X)
            \ar{dr}{w_2^s|} &&
         & \pi_2\X
            \rar{w_2^s} 
            \dar{\foliate_{\bc}}
            %\dar[two heads]{\unrot} 
        & \Z/2\rar\dar[equals]
        & \pi_1\Fr\X
            \dar[two heads][swap]{mod\conjbc}
            \rar[two heads]
        & \pi_1\X
            \dar[two heads][swap]{mod\conjbc}
        \\
        \pi_2(\Lambda\X;c)
            \rar{\delta_2} 
            \uar[two heads]{\pi_2\ev_e}  
        & \Z/2
            \rar{\rot_{\nu c}}%\dar[tail,two heads] 
        & \pi_1(\Lambda\Fr\X;\nu c)
            \rar{\pi_1pr} %\dar[tail,two heads] 
        & \pi_1(\Lambda\X;c)
            \rar{\delta_1} %\dar[tail,two heads]%{(\pi_1\ev_e\tm\unrot)\circ\pi_1p} 
        & \Z/2 
            \rar{\tw_{\nu c}}%\dar[tail,two heads] 
        & \pi_0\Lambda\Fr\X
            \rar[two heads]{\pi_0pr} %\dar[tail,two heads] 
        & \pi_0\Imm
            % \dar[tail,two heads]{\pi_0p} 
    \end{tikzcd}
    \end{equation}
where the horizontal sequences are exact.
\end{lem}

\begin{proof}
    We consider the commutative diagram
    \[\begin{tikzcd}[row sep=small]
            \Omega SO_4\rar[hook]\dar & \Omega \Fr\X \rar{\Omega pr}\dar & \Omega \X\dar\\
            \Lambda SO_4\rar[hook]\dar & \Lambda \Fr\X \rar{\Lambda pr}\dar & \Lambda \X\dar{\ev_e}\\
            SO_4\rar[hook] & \Fr\X \rar{pr} & \X
        \end{tikzcd}
    \]
    where each row and column is a fibration sequence. For $\Omega \Fr\X$ and $\Lambda \Fr\X$ the basepoint is $\nu c$, whereas for $\Omega \X$ and $\Lambda \X$ it is $c$. 
    Now consider the long exact sequences in homotopy groups:
    \begin{equation}\label{eq:big-diag}
    \begin{tikzcd}[scale cd=0.9]
        \pi_3\X
            \rar\dar
        & \pi_2 SO_4=0\rar\dar
        & \pi_2(\Fr\X)
            \rar[tail]{\pi_2 pr}
            \dar
        & \pi_2\X
            \rar{\delta_{pr}=w_2^s}
            \dar[swap]{\foliate_{\bc}} 
        & \pi_1 SO_4=\Z/2
            \dar[tail,two heads] \\
        \pi_2(\Lambda \X;\bc)
            \rar{\delta_2}
            \dar[swap]{\pi_2\ev_e} 
        & \pi_1(\Lambda SO_4) 
            \rar{\rot_{\nu c}}\dar[tail,two heads]
        & \pi_1(\Lambda \Fr\X;\nu\bc)
            \rar{\pi_1\Lambda pr}
            \dar[swap]{\pi_1\ev_e}
        & \pi_1(\Lambda \X;\bc)
            \rar{\delta_1}
            \dar[swap]{\pi_1\ev_e}
        & \pi_0(\Lambda SO_4) 
            \dar \\
        \pi_2\X
            \rar{\delta_{pr}=w_2^s}
            \dar[swap]{b\mapsto b-\bc\cdot b}
        & \pi_1(SO_4)=\Z/2\rar\dar{0}
        & \pi_1(\Fr\X)
            \rar[two heads]{\pi_1pr}
            \dar{\conjbc}
        & \pi_1\X
            \rar
            \dar{\conjbc}
        & \pi_0SO_4=0 \dar{0}\\
        \pi_2\X
            \rar{\delta_{pr}=w_2^s}
            \dar{\foliate_{\bc}}
        &\pi_1(SO_4)=\Z/2
            \rar
            \dar[tail,two heads]
        & \pi_1(\Fr\X)
            \rar[two heads]{\pi_1pr}
            \dar[two heads]
        & \pi_1\X
            \rar
            \dar[two heads]
        & \pi_0SO_4=0\\
        \pi_1(\Lambda \X;\bc)
            \rar{\delta_1}
        & \pi_0(\Lambda SO_4) 
            \rar{\tw_{\nu c}}
        & \pi_0(\Lambda \Fr\X)
            \rar[two heads]{\pi_0 \Lambda pr}
        & \pi_0(\Lambda \X)
        & 
    \end{tikzcd}
    \end{equation}
    At the bottom we have a sequence of sets, and maps from groups to sets depend on the basepoint. 
    By Lemma~\ref{lem:FrX} the map $\delta_{pr}\colon\pi_2\X\to\Z/2$ agrees with $w_2^s$. 
    
    By the commutativity of the $\delta_{pr}$/$\delta_2$-square in the leftmost column, $\delta_2=\delta_{pr}\circ \pi_2\ev_e=w_2^s\circ \pi_2\ev_e$. Since $\im(\pi_2\ev_e)=\Fix_{\bc}(\pi_2\X)$ by Corollary~\ref{cor:Lambda}, we have the desired triangle for $\delta_2$.
    
    The bottom two rows give the desired diagram for $\delta_1$.
\end{proof}

Since $\pi_2\ev_e$ is surjective, it follows that $\delta_2$ --- and thus $\rot_{\nu c}$ --- is completely determined by $w_2^s$. However, for $\delta_1$ and $\tw_{\nu c}$ the situation is not yet clear. We can say more using Corollary~\ref{cor:w2-FrX}.

\begin{lem}
\label{lem:lem2}
    Fix an oriented 4-manifold $\X$ and $\nu c\colon\nu\S^1\hra \X$.
    \begin{itemize}
        \item 
    If $w_2^s\neq0$, then $\tw_{\nu c}=\nu c$. 
        \item 
    If $w_2^s=0$, $w_2^h\neq0$, then $\tw_{\nu c}=\nu c$ if and only if there exists $y\in\Fix_{\bc}(\pi_1\X)$ with $w_2^h([y\tm\bc])\neq0$.
        \item 
    If $w_2^h=0$, then $\tw_{\nu c}\neq\nu c$.
    \end{itemize}
\end{lem}
\begin{proof}
    Looking at the diagram \eqref{eq:lem1}, if $w_2^s\neq0$ then for any choice of $c$ the map $\delta_1$ has to be surjective as well. Thus $\tw_{\nu c}=\nu c$.

    On the other hand, if $w_2^h=0$ then $\Z/2\ira \pi_1\Fr\X\sra\pi_1\X$ and $\Z/2\ira H_1(\Fr\X)\sra H_1X$ by Corollary~\ref{cor:w2-FrX}, so the element $z\in\Z/2$ survives abelianization. This means that modding out by $x\sim \bc \cdot x$ cannot kill the element $z$, so $\tw_{\nu c}\neq\nu c$.
    
    Finally, when $w_2^s=0$, $w_2^h\neq0$, the conjugation action by $\bc$ kills the element $z\in\Z/2\ira\pi_1(\Fr\X)$ if and only if we can find $y\in\pi_1\X$ that commutes with $\bc\in\pi_1\X$ so that $w_2^h([y\tm\bc])=[\wt{y},\wt{\bc}]\neq0$.
    % As in Corollary~\ref{cor:w2-dim4} the latter condition means that the torus $y\tm\bc\colon\S^1\tm\S^1\to\X$ cannot be framed.
\end{proof}

\begin{rem}
     Note that when $w_2^h\neq0$ both options can occur in the same $X$ for different $\bc\in\pi_1\X$. For example, for the trivial element $\bc=1$ we have $w_2^h([y\tm 1])=0$ for all $y$, and $\tw_{\nu c}\neq\nu c$.
\end{rem}

We can now deduce the result stated in the introduction.

\begin{cor}[Theorem~\ref{thm-intro:framed-imm}]
\label{cor:framed-imm}
    For an oriented  smooth 4-manifold $\X$ and a framed embedded circle $\nu c\colon\nu\S^1=\S^1\tm\D^3\hra\X$, we have the following exact sequences.
\begin{enumerate}
 \item (totally nonspin) Assume $w_2^s\neq0$.
    \begin{enumerate}
\item 
    If there exists $b\in\pi_2\X$ with $w_2^s(b)=1$ and $b=\bc\cdot b$, (that is, there is an immersed sphere in $\X$ of odd self-intersection, and fixed by the action of $\bc$), then
    \[
    \begin{tikzcd}[column sep=small]
        \quad%\Z/2
            %\rar{0}
        & \pi_1(\Imm^{\fr};\nu c)
            \rar[tail] 
        & \pi_1(\Imm;c)
            \rar[two heads] 
        & \Z/2, 
            %\rar{0}
        & \nu c=\tw_{\nu c}.
    \end{tikzcd}
    \]
\item 
    Otherwise,
    \[
    \begin{tikzcd}[column sep=small]
        \Z/2
            \rar[tail]{\rot_{\nu c}}
        & \pi_1(\Imm^{\fr};\nu c)
            \rar
        & \pi_1(\Imm;c)
            \rar[two heads]
        & \Z/2, 
            % \rar{0}
        & \nu c=\tw_{\nu c}.
    \end{tikzcd}
    \]
    \end{enumerate}
\item 
(spin or almost spin)
Assume $w_2^s=0$.
    \begin{enumerate}
\item
If there exists $y\in\Fix_{\bc}(\pi_1\X)$ such that $w_2^h([y\tm\bc])\neq0$ (that is, there is an immersed torus in $\X$ of odd self-intersection, and that contains $\bc$), then
    \[\begin{tikzcd}[column sep=small]
        \Z/2
            \rar[tail]{\rot_{\nu c}}
        & \pi_1(\Imm^{\fr};\nu c)
            \rar
        & \pi_1(\Imm;c)
            \rar[two heads] 
        & \Z/2, 
            %\rar
        & \nu c=\tw_{\nu c}.
    \end{tikzcd}
    \]
\item 
    Otherwise,
    \[\begin{tikzcd}[column sep=small]
        \Z/2
            \rar[tail]{\rot_{\nu c}}
        & \pi_1(\Imm^{\fr};\nu c)
            \rar[two heads] 
        & \pi_1(\Imm;c),
            %\rar{0} 
        &\qquad\hfill
        & \nu c\neq\tw_{\nu c}.
    \end{tikzcd}
    \]
    \end{enumerate}
\end{enumerate} 
    Moreover, $\rot_{\nu c}$ splits if $w_2=0$ ($\X$ is spin). More generally, it splits whenever the extension $\Z/2\to\Fix_{\nu c}(\pi_1\Fr\X)\sra\Fix_c(\pi_1\X)$ does.
\end{cor}
\begin{proof}
    By Lemma~\ref{lem:lem1} the map $\rot_{\nu c}$ is trivial if and only if there exists $b\in\Fix_{\bc}(\pi_2\X)$ with $w_2^s(b)=1$. Thus, it is trivial precisely in the case \textit{(1a)}, where the equivalent statement in the parenthesis was given in Corollary~\ref{cor:w2-dim4}. By Lemma~\ref{lem:lem2} the map $\tw_{\nu c}$ is trivial as soon as $w_2^s\neq0$, implying \textit{(1)}. Similarly, \textit{(2)} follows from the rest of Lemma~\ref{lem:lem2}.

    It remains to prove the splitting statement for $\rot_{\nu c}$ for the cases \textit{(1b)} and \textit{(2)}. In \eqref{eq:big-diag} consider the second sequence of vertical arrows from the top, whose images form the short exact sequence $\Z/2\ira\Fix_{\nu c}(\pi_1\Fr\X)\sra\Fix_c(\pi_1\X)$. By the commutativity of the diagram we see that $\rot_{\nu c}$ splits if this extension does, giving the second splitting claim. On the other hand, this extension is obtained by restricting to fixed point subgroups the extension $\Z/2\ira\pi_1\Fr\X\sra \pi_1\X$.
    By Lemma~\ref{lem:FrX} whether the last extension splits is determined by $w_2$.
\end{proof}

Note that for the splitting of $\rot_{\nu c}$ it suffices to frame only those loops that are fixed by $\bc$. This is less than splitting $\Z/2\ira\pi_1\Fr\X\sra \pi_1\X$, that is, framing all loops, as given by a spin structure.

%%%~~~~~~~~~~~~~~~~~~%%%
\section{Embeddings}
\label{sec:framed-emb}

In this section we study the difference of the framed to the nonframed case in the setting of embeddings, and prove Theorem~\ref{thm-intro:framed-emb}.

We have the commutative diagram of fibrations
\begin{equation}\label{eq:diag-emb-imm}
    \begin{tikzcd}[row sep=small]
        \forg^{-1}(c)\rar[hook]\dar[sloped]{\sim} & \Emb(\nu\S^1,\X) \rar{\forg}\dar & \Emb(\S^1,\X)\dar[near start]{\incl},\\
        \forg_{\Imm}^{-1}(c)\rar[hook] & \Imm(\nu\S^1,\X)\rar{\forg_{\Imm}} & \Imm(\S^1,\X)
        % \dar{\sim},\\
        % \Lambda SO_3\rar[hook] & \Lambda \Fr\X \rar{\Lambda pr_1} & \Lambda \S \X.
    \end{tikzcd}
\end{equation}
A point in the top fibre $\forg^{-1}(c)$ is given by a framing of the basepoint $c\colon\S^1\hra \X$ (that is, trivialisation of its normal bundle), and this agrees with the fibre $(\forg_{\Imm})^{-1}(c)$ at the bottom. By general position we have an isomorphism $\pi_0\incl$ and a surjection $\pi_1\incl$.

Combining this with \eqref{eq:diag-imm-Lambda} and \eqref{eq:lem1}, we obtain the following maps between two horizontal exact sequences, and two commutative triangles:
\[\begin{tikzcd}[column sep=10pt, nodes={scale=0.9}]
        \pi_2(\Emb;c)
            \rar{\delta^{\Emb}_2}\dar[swap]{\pi_2\incl}
        & \pi_1 \forg^{-1}(c)
            \rar\dar[tail,two heads]  
        & \pi_1(\Emb^{\fr};\nu c)
            \rar{\pi_1\forg}\dar 
        & \pi_1(\Emb;c)
            \rar{\delta^{\Emb}_1}\dar[two heads][swap]{\pi_1\incl}
        & \pi_0 \forg^{-1}(c) 
            \rar\dar[tail,two heads]  
        & \pi_0\Emb^{\fr}
            \dar \rar[two heads]{\pi_0\forg} 
        & \pi_0\Emb
            \dar[tail,two heads]{\pi_0\incl}\\
        \pi_2(\Imm;c)
            \rar{\delta^{\Imm}_2} \dar[two heads][swap]{\pi_2\ev_e}  
        & \Z/2
            \rar{\rot_{\nu c}}
        & \pi_1(\Imm^{\fr};\nu c)
            \rar{\pi_1\forg_{\Imm}} 
        & \pi_1(\Imm;c)
            \rar{\delta^{\Imm}_1}%\dar[tail,two heads][swap]{\pi_1\ev_e\tm\unrot} 
        & \Z/2
            \rar{\tw_{\nu c}}
        & \pi_0\Imm^{\fr}
            \rar[two heads]{\pi_0\forg_{\Imm}} 
        & \pi_0\Imm
            \dar[tail,two heads]{\pi_0p} \\
        \Fix_{\bc}(\pi_2\X)
            \ar{ur}[swap]{w_2^s}
        && 
        & \faktor{\pi_2\X}{b=\bc\cdot b}
            \ar{ur}[swap]{w_2^s}\uar[tail]{\foliate_{\bc}}
        && 
        & \faktor{\pi_1\X}{x\sim y\cdot x}
    \end{tikzcd}
\]
By the commutativity of this diagram and the fact that $\pi_i(\forg^{-1}(c))\cong\pi_i(\forg_{\Imm}^{-1}(c))\cong\Z/2$ for $i=0,1$, we get a description of $\delta_1^{\Emb}$ and $\delta_2^{\Emb}$. Namely,
\[
    \im(\delta_1^{\Emb})=\im(\delta_1^{\Imm})\leq\Z/2,
\]
so the (non)triviality of $\tw_{\nu c}$ is determined as in Corollary~\ref{cor:framed-imm}. In particular, we always have
\[
    \pi_0\Emb^\fr\cong \pi_0\Imm^{\fr}.
\]
However, since $\pi_2\incl\colon\pi_2(\Emb;c)\to \pi_2(\Imm;c)$ is in general not surjective, we cannot yet make a conclusion about the (non)triviality of $\rot_{\nu c}$. We need to first determine the source of $\delta_2^{\Emb}$, or equivalently, the image of $\pi_2\incl$. 
For this we rely on the following result from our previous work \cite{K-Dax}, using the homomorphism $\dax^{whisk}_c$ from~\eqref{eq:dax-whisk}. 
\begin{lem}\label{lem:dax-whisk}
    The image of $\pi_2\ev_e\circ\pi_2\incl$ consists precisely of those $b\in\pi_2\X$ that have $\dax^{whisk}_c(b)=0$.
\end{lem}
 \begin{proof}
    By the long exactness sequence of the fibration $\incl\colon\Emb\to\Imm$, the image of $\pi_2\incl$ is the same as the kernel of the connecting map $\delta_{d\incl}$. This kernel was computed in \cite[Thm.\ 4.5, par.\ after Lem.\ 4.9]{K-Dax}, which says that 
    \[\begin{tikzcd}[column sep=2cm]
        \im(\pi_2\incl)=\ker(\delta_{d\incl})
            \ar{r}{\pi_2\ev_e} 
        & \Fix_{\bc}(\pi_2\X)
            \rar{\dax^{whisk}_c}
        & \faktor{\Z[\pi_1\X\sm\{1\}]}{\dax_u(\pi_3(\X\sm\D^4))}
    \end{tikzcd}
    \]
    is an exact sequence, for $\dax^{whisk}_c\coloneqq\Da\circ\delta^{whisk}_c$. Therefore, $\im(\pi_2\ev_e\circ\pi_2\incl)=\ker(\dax^{whisk}_c)$.
\end{proof}

Combining this with the last diagram we obtain the following result.
\begin{cor}\label{cor:framed}
    For an oriented 4-manifold $\X$ and $\nu c\colon\nu\S^1\hra \X$ there is an exact sequence of groups and sets:
    \[\begin{tikzcd}[column sep=17pt]
        \ker(\dax^{whisk}_c)%\{b\in\pi_2\M\colon \dax^{whisk}_c(b)=0\}
            \rar{w_2^s|} 
        & \Z/2
            \rar{\rot_{\nu c}}
        & \pi_1(\Emb^{\fr};\nu c)
            \rar{\pi_1\forg} 
        & \pi_1(\Emb;c)
            \rar{\delta_1^{\Emb}}%\dar[two heads][swap]{\unrot}
        & \Z/2
            \rar{\tw_{\nu c}} 
        & \pi_0\Emb^{\fr}
            \rar[two heads]{\pi_0\forg} 
        & \pi_0\Emb.
    \end{tikzcd}
    \]
\end{cor}

From this and Theorem~\ref{thm-intro:framed-imm} we immediately deduce Theorem~\ref{thm-intro:framed-emb}.

% References

\printbibliography[heading=bibintoc]

\end{document}